\documentclass[12pt]{article}

\usepackage{amsthm, amsmath, amssymb, verbatim,graphicx}

\numberwithin{equation}{section}

\theoremstyle{plain}	     
\newtheorem{thm}{Theorem}[section] 
\newtheorem{cor}[thm]{Corollary}
\newtheorem{lem}[thm]{Lemma}
\newtheorem{prop}[thm]{Proposition}
     
\theoremstyle{definition} 
\newtheorem{defn}[thm]{Definition}
\newtheorem{exmp}[thm]{Example}

\theoremstyle{remark} 
\newtheorem{rem}[thm]{Remark}
	
\newcommand{\disp}{\displaystyle}

\newcommand{\sn}{\operatorname{sn}}
\newcommand{\cn}{\operatorname{cn}}
\newcommand{\dn}{\operatorname{dn}}

\begin{document}
\title{The basis property of generalized Jacobian elliptic functions
\footnote{This work was supported by MEXT/JSPS KAKENHI Grant (No. 24540218).}}
\author{Shingo Takeuchi \\
Department of Mathematical Sciences\\
Shibaura Institute of Technology
\thanks{307 Fukasaku, Minuma-ku,
Saitama-shi, Saitama 337-8570, Japan. \endgraf
{\it E-mail address\/}: shingo@shibaura-it.ac.jp \endgraf
}}
\date{}

\maketitle

\begin{abstract}
The Jacobian elliptic functions are generalized
to functions including the generalized trigonometric functions.
The paper deals with the basis property of the sequence of 
generalized Jacobian elliptic functions in any Lebesgue space. 
In particular, it is shown that the sequence of the classical 
Jacobian elliptic functions is a basis in any Lebesgue space
if the modulus $k$ satisfies $0 \le k \le 0.99$. 
\end{abstract}


\section{Introduction}

The Jacobian elliptic function $\sn(x,k)$ and 
the complete elliptic integral of the first kind $K(k)$
play important roles 
in expressing exact solutions of, 
for example, the pendulum equation $u''+\lambda \sin{u}=0$,
a typical bistable equation $u''+\lambda u(1-u^2)=0$, and so on. 

Now we will propose new generalization of $\sn(x,k)$ and $K(k)$.
For constants $p,\ q \in (1,\infty)$ and $k \in [0,1)$, 
we define a \textit{generalized Jacobian 
elliptic function} $\sn_{pq}(x,k):[0,K_{pq}(k)] \to [0,1]$ with 
a modulus $k$ as
$$x=\int_0^{\sn_{pq}(x,k)}
=\frac{dt}{(1-t^q)^{\frac1p}(1-k^qt^q)^{\frac1{p'}}},$$
where $p'=p/(p-1)$ and  
$$K_{pq}(k)=\int_0^1 \frac{dt}{(1-t^q)^{\frac1p}(1-k^qt^q)^{\frac1{p'}}}.$$
We extend the domain of $\sn_{pq}(x,k)$ 
to $\mathbb{R}$ so that we obtain a $4K_{pq}(k)$-period function
like the sine function,
and call the extended function $\sn_{pq}(x,k)$ again.
Then, $\sn_{22}(x,k)=\sn(x,k)$ and $K_{22}(k)=K(k)$ when $p=q=2$;
and $\sn_{pq}(x,0)=\sin_{pq}{x}$ and $K_{pq}(0)=\pi_{pq}/2$ when $k=0$,
where $\sin_{pq}{x}$ is the generalized trigonometric function 
and $\pi_{pq}$ is the half period of $\sin_{pq}{x}$,
which will be introduced in Section \ref{sec:GJEF} below.
Therefore, $\sn_{pq}(x,k)$ is also generalization of 
both $\sn(x,k)$ and $\sin_{pq}{x}$.

In the previous paper \cite{T}, 
the author proposed another generalization of $\sn(x,k)$ and $K(k)$,
and applied them to bifurcation problems for $p$-Laplacian.
As we will see in Section \ref{sec:GJEF}, 
$\sn_{pq}(x,k)$ and $K_{pq}(k)$ above are defined in a slightly 
different way from those in \cite{T}, but $\sn_{pq}(x,k)$ 
also satisfies the following equation involving $p$-Laplacian nevertheless.
$$(|u'|^{p-2}u')'+\frac{(p-1)q}{p}|u|^{q-2}u
(1+(p-1)k^q-pk^q|u|^q)(1-k^q|u|^q)^{p-2}=0.$$
While the generalization of $K(k)$ of \cite{T} 
converges to a finite value as $k \to 1$ when $p>2$, 
the $K_{pq}(k)$ diverges to $\infty$ as $k \to 1$ for any $p>1$. In this sense, 
$\sn_{pq}(x,k)$ has closer properties to $\sn(x,k)$ than the function 
defined in \cite{T}. 

In the present paper, we will show the basis property of functions
\begin{equation}
\label{eq:functions}
f_n(x,k)=\sn_{pq}(2nK_{pq}(k)x,k),\quad n=1,2,\ldots,
\end{equation}
which means that 
the family of these functions is a basis in Banach spaces.
Here, a sequence $\{\varphi_n\}$ in a Banach space $X$ is 
called a basis for $X$
if for every $u \in X$ there exists a unique sequence of scalars 
$\{\alpha_n\}$ such 
that $u=\sum_{n=1}^{\infty}\alpha_n \varphi_n$ in the strong sense. 
In general, when we try to find an approximation of a given function
by a family of functions $\{\varphi_n\}$, it is desirable that
$\{\varphi_n\}$ is a basis which approximates to the function 
with convergence of higher order as possible. 
Concerning this, for example, we have known an interesting study
\cite{BL} of Boulton and Lord. They study 
the best index $q$ for which $\{\sin_{q}{(n\pi_{q}x)}\}$ approximates
well to the solution of $p$-Poisson problem, where $\sin_q{x}=\sin_{qq}{x}$
and $\pi_{q}=\pi_{qq}$. 
The basis property is quite fundamental
to such a stimulating problem.
     
When $p=q=2$, the sequence \eqref{eq:functions} is the family 
of Jacobian elliptic functions $\{\sn(2nK(k)x,k)\}$. 
In this case, Craven \cite{C} proves that if the modulus $k$ 
satisfies $0 \leq k \leq 0.99$,
then the sequence is complete in $L^2(0,1)$. Since the sequence
is not orthogonal, we have no guarantee of its basis property.

On the other hand, when $k=0$, the sequence \eqref{eq:functions}
is the family of generalized trigonometric functions 
$\{\sin_{pq}(n\pi_{pq}x)\}$. On the sequence for $p=q$, 
Binding et al.\,\cite{BBCDG} first studied the basis property.
Recently, Edmunds et al.\,\cite{EGL}
show that if, for example, $p'/q<4/(\pi^2-8)$, then the sequence is 
a basis in $L^{\alpha}(0,1)$ for any $\alpha \in (1,\infty)$.

This paper deals with the basis property of \eqref{eq:functions}
for general $p,\ q \in (1,\infty)$ and $k \in [0,1)$. Our results involve results 
of \cite{C} when $p=q=2$ and \cite{EGL} when $k=0$.

\begin{thm}
\label{thm:Main}
Let $p,\ q \in (1,\infty)$ and $r=\max\{p',q\}$. 
If 
\begin{equation}
\label{eq:Alpha}
\frac{1}{q}B\left(\frac{1}{r},\frac{1}{r}\right) <\frac{8}{\pi^2-8},
\end{equation}
then $\{f_n(x,k)\}$ forms a Riesz basis of $L^2(0,1)$ and 
a Schauder basis of $L^\alpha(0,1)$ for any 
$\alpha \in (1,\infty)$ 
when $k=0$ or
\begin{equation}
\label{eq:K}
\frac{\tanh_r^{-1}{k^{\frac{q}{r}}}}{k^{\frac{q}{r}}} \leq 
\frac{8q}{\pi^2-8}
B\left(\frac{1}{r},\frac{1}{r}
\right)^{-1},
\end{equation}
where $B(x,y)=\int_0^1 t^{x-1}(1-t)^{y-1}\,dt$ is the Beta function
and $\tanh_r^{-1}{x}$ is a \textit{generalized inverse hyperbolic function}
$$\tanh_r^{-1}{x}=\int_0^x \frac{dt}{1-t^r}.$$
\end{thm}

\begin{rem}
The function $\tanh_r^{-1}{x}/x$ is a monotone increasing function
from $(0,1)$ onto $(1,\infty)$.
\end{rem}

From Theorem \ref{thm:Main}, 
we obtain an improvement of Craven's result stated in Remark
of \cite[Theorem 2]{C}.

\begin{cor}
\label{cor:main}
If $1<p' \leq q<\infty$, then $\{f_n(x,k)\}$
forms a Riesz basis of $L^2(0,1)$ and 
a Schauder basis of $L^\alpha(0,1)$ for any 
$\alpha \in (1,\infty)$
when $k=0$ or
\begin{equation*}
\frac{\tanh_q^{-1}{k}}{k} \leq 
\frac{8q}{\pi^2-8}
B\left(\frac{1}{q},\frac{1}{q}
\right)^{-1}.
\end{equation*}
In particular, 
the sequence of Jacobian elliptic functions $\{\sn(2nK(k)x,k)\}$ 
does so when $0 \leq k \leq 0.99$.
\end{cor}

We will give another corollary of Theorem \ref{thm:Main},
whose conditions are verified easier than 
\eqref{eq:Alpha} and \eqref{eq:K}. 

\begin{cor}
\label{thm:main}
Let $p,\ q \in (1,\infty)$ and $r=\max\{p',q\}$. 
If 
\begin{equation}
\label{eq:alpha}
\frac{r}{q} <\frac{4}{\pi^2-8},
\end{equation}
then $\{f_n(x,k)\}$  forms a Riesz basis of $L^2(0,1)$ and 
a Schauder basis of $L^\alpha(0,1)$ for any 
$\alpha \in (1,\infty)$ when 
\begin{equation}
\label{eq:k}
0 \leq k< \left[1-\left\{\frac{(\pi^2-8)r}{4q}\right\}^{r}\right]^{\frac1q}.
\end{equation}
\end{cor}

\begin{rem}
(i) If $p' \leq q$, i.e., $r=q$, then \eqref{eq:Alpha} and \eqref{eq:alpha}
hold.
(ii) Case $k=0$ in \eqref{eq:k} corresponds to the result of 
\cite[Theorem 4.4]{EGL}.
(iii) When $r=q=2$, the value of the right-hand side of \eqref{eq:k} 
is about $0.88$, which is not so satisfactory as
$0.99$ of Corollary \ref{cor:main}. However, we can check 
\eqref{eq:k} much easier than \eqref{eq:K}.
\end{rem}

The paper is organized as follows. In Section \ref{sec:PB} we give a 
summary of general properties of bases in Banach spaces.
In Section \ref{sec:GJEF} we recall the generalized 
trigonometric functions and introduce new generalization of 
Jacobian elliptic functions. 
In Section \ref{sec:PsK} we observe properties of 
the generalized Jacobian elliptic function $\sn_{pq}(x,k)$
and its quarter period $K_{pq}(k)$.
To show that the sequence 
\eqref{eq:functions} is a basis in $L^\alpha(0,1)$ for any $\alpha \in (0,1)$,
we depend on the strategy of Binding et al. \cite{BBCDG} and 
Edmunds et al. \cite{EGL}.
Our main device is a linear mapping $T$ of $L^\alpha(0,1)$,
satisfying $Te_n=f_n$, where $e_n=\sin{(n\pi x)}$, and decomposing into
a linear combination of certain isometries. In Section \ref{sec:TOT}
we show that $T$ is a bounded operator for $p \in (1,\infty)$. 
Section \ref{sec:BIT} is devoted to the proof of 
boundedness of the inverse for the
ranges \eqref{eq:Alpha} and \eqref{eq:K}.


\section{Properties of Bases}
\label{sec:PB}

In this section we will give a summary of properties of
bases in Banach spaces.
For details, we can refer to Gohberg and Kre\u{\i}n \cite{GK},
Higgins \cite{H}, and Singer \cite{S}.

A sequence $\{x_n\}$ in an infinite dimensional Banach space $X$
is called a \textit{basis of $X$} 
if for every $x \in X$ there exists
a unique sequence of scalars $\{\alpha_n\}$ such that 
$x=\sum_{i=1}^{\infty} \alpha_i x_i$ (i.e., such that $\lim_{n \to \infty}
\|x-\sum_{i=1}^n \alpha_i x_i \|=0$).
A basis $\{x_n\}$ of a topological linear space $U$ is said to be 
a \textit{Schauder basis} of $U$, if all coefficient functionals $f_n,\
n=1,2,\ldots$, are continuous on $U$.

\begin{prop}
\label{prop:schauder}
Every basis of a Banach space is a Schauder basis of this space.
\end{prop}

\begin{proof}
See \cite[Theorem 3.1, p.20]{S}. 
\end{proof}

\begin{defn}
A basis $\{x_n\}$ of a Banach space $X$ is said to be

(a) a \textit{Bessel basis}, if 
$$\sum_{i=1}^{\infty}\alpha_i x_i\ \mbox{is convergent} \Rightarrow
 \sum_{i=1}^{\infty}|\alpha_i|^2<\infty,$$
i.e., there exists a constant $c>0$ such that we have
$$c \sqrt{\sum_{i=1}^n|\alpha_i|^2}
\le \left\|\sum_{i=1}^n \alpha_i x_i\right\|$$
for all finite sequences of scalars $\alpha_1,\ldots,\alpha_n$.

(b) a \textit{Hilbert basis}, if 
$$\sum_{i=1}^{\infty}|\alpha_i|^2<\infty\ \Rightarrow
\sum_{i=1}^{\infty}\alpha_i x_i\ \mbox{is convergent},$$
i.e., there exists a constant $C>0$ such that we have
$$\left\|\sum_{i=1}^n \alpha_i x_i\right\|
\le C \sqrt{\sum_{i=1}^n|\alpha_i|^2}$$
for all finite sequences of scalars $\alpha_1,\ldots,\alpha_n$.

(c) a \textit{Riesz basis}, if it is both a Bessel basis and a Hilbert 
basis,
i.e., there exist two constants $c>0$ and $C>0$ such that we have
$$c \sqrt{\sum_{i=1}^n|\alpha_i|^2}
\le \left\|\sum_{i=1}^n \alpha_i x_i\right\|
\le C \sqrt{\sum_{i=1}^n|\alpha_i|^2}$$
for all finite sequences of scalars $\alpha_1,\ldots,\alpha_n$.
\end{defn}

\begin{exmp}
\label{exmp:singer343}
In the space $X=L^p(-\pi,\pi),\ p \in (1,\infty)$, 
the sequence $\{x_n\}$, where
$$x_0(t)=\dfrac12, \quad x_{2n-1}(t)=\sin{nt} ,\quad
x_{2n}(t)=\cos{nt} \quad (t \in [-\pi,\pi],\ n=1,2,\ldots)$$
is a bounded Bessel basis if $p \ge 2$ and a bounded
Hilbert basis if $1<p \le 2$. In particular, 
it is a Riesz basis if $p=2$. 
\end{exmp}

\begin{proof}
See \cite[Example 11.1, pp.342--345]{S}.
\end{proof}

We call two sequences $\{\phi_n\}$ and 
$\{\psi_n\}$ in Banach space $X$ \textit{equivalent}
if there exists a linear homeomorphism
(i.e., bounded, linear and invertible operator) $T$
on $X$ such that $\psi_n=T(\phi_n)$ for every $n$.
Note that by `invertible' we mean that $T^{-1}$ exists and 
is bounded on all of $X$. 

\begin{prop}
\label{prop:higgins75}
If a sequence $\{\psi_n\}$
in a Banach space is equivalent to a basis
$\{\phi_n\}$, it too is a basis.
\end{prop}

\begin{proof}
See \cite[Lemma, p.\,75]{H}.
\end{proof}

Furthermore, we have

\begin{prop}[Bari]
\label{thm:bari311}
Let $H$ be a Hilbert space. The following assertions are equivalent.

$(\rm{a})$ The sequence $\{\phi_n\}$ forms a basis of 
$H$, equivalent to an orthonormal basis.

$(\rm{b})$ The sequence $\{\phi_n\}$ is complete and 
a Riesz basis in $H$.
\end{prop}

\begin{proof}
See the first and third assertions of \cite[Theorem 2.1, pp.310--311]{GK}.
\end{proof}


\section{Generalized Functions}
\label{sec:GJEF}

This section is devoted to the definitions of
 two kinds of generalized functions.

\subsection{Generalized Trigonometric Functions}
\label{ssec:GTF}

\textit{Generalized Trigonometric functions} were introduced in 
1879 by E.\,Lundberg (see 
Lindqvist and Peetre \cite[pp.113-141]{LP2}).
After that, these functions have been developed mainly by 
Elbert \cite{E}, Lindqvist 
\cite{Li}, Dr\'{a}bek and Man\'{a}sevich \cite{DM}, 
and Lang and Edmunds \cite{LE}.

For any constants $p,\ q \in (1,\infty)$, we define $\pi_{pq}$ by 
$$\pi_{pq}=2\int_0^1 \frac{dt}{(1-t^q)^{\frac1p}}
=\frac{2}{q}B\left(1-\frac{1}{p},\frac{1}{q}\right)
=\frac{2\Gamma(1-\frac1p)\Gamma(\frac1q)}{q\Gamma(1-\frac1p+\frac1q)},$$
where $B$ and $\Gamma$ are the Beta and Gamma functions, respectively. 
Then, for any $x \in [0,\pi_{pq}/2]$ 
we define $\sin_{pq}{x}$ by
$$x=\int_0^{\sin_{pq}{x}} \frac{dt}{(1-t^q)^{\frac1p}}.$$
Clearly, $\sin_{pq}{x}$ is an increasing function in $x$ 
from $[0,\pi_{pq}/2]$ onto $[0,1]$. 
We extend the domain of $\sin_{pq}{x}$ to 
$[0,\pi_{pq}]$ by $\sin_{pq}{x}=\sin_{pq}(\pi_{pq}-x)$,
and furthermore, to the whole of $\mathbb{R}$ by 
$\sin_{pq}{(x+\pi_{pq})}=-\sin_{pq}{x}$, so that $\sin_{pq}{x}$ has
$2\pi_{pq}$-periodicity. 
We can see that $\pi_{22}=\pi$ and $\sin_{22}{x}=\sin{x}$. Moreover, 
the function $y=\sin_{pq}{x}$
satisfies that $y,\ |y'|^{p-2}y' \in C^1(\mathbb{R})$, 
and $y \in C^2(\mathbb{R})$ if $1<p \leq 2$.

We agree that $\pi_p$ and $\sin_p{x}$ denote $\pi_{pp}$ and $\sin_{pp}{x}$
when $p=q$, respectively. In that case, we can also refer to 
\cite{DEM,Do,DoR,DKT,E,Li}.

Using $\sin_{pq}{x}$, for $x \in [0,\pi_{pq}/2]$ we also define 
\begin{equation}
\label{eq:cos}
\cos_{pq}{x}=(1-\sin_{pq}^q{x})^{\frac1q}.
\end{equation}
Clearly, $\cos_{pq}{x}$ is a decreasing function in $x$ 
from $[0,\pi_{pq}/2]$ onto $[0,1]$.
We extend the domain of $\cos_{pq}{x}$ 
to $[-\pi_{pq}/2,\pi_{pq}/2]$ by $\cos_{pq}{x}=\cos_{pq}{(-x)}$,
and furthermore, 
to the whole of $\mathbb{R}$ in the same way as $\sin_{pq}{x}$.
Then, $\cos_{pq}{x}$ has $2\pi_{pq}$-periodicity. 
We can see that $\cos_{22}{x}=\cos{x}$.
An analogue of $\tan{x}$ is obtained by defining
$$\tan_{pq}{x}=\frac{\sin_{pq}{x}}{\cos_{pq}{x}}$$
for those values of $x$ at which $\cos_{pq}{x}\neq 0$.
This means that $\tan_{pq}{x}$ is defined for all $x \in \mathbb{R}$ except
for the points $(k+1/2)\pi_{pq}\ (k \in \mathbb{Z})$.
We denote by $\cos_{p}{x}$ and $\tan_p{x}$ as for the case $\sin_p{x}$.
The functions $\sin_p{x}$ and $\cos_p{x}$ are useful for Pr\"{u}fer
transformation of half-linear differential equations.
For this, see \cite{Do,DoR,E,N}.

These functions satisfy, for $x \in (0,\pi_{pq}/2)$,
\begin{gather}
\cos_{pq}^q{x}+\sin_{pq}^q{x}=1,\label{eq:cs} \\
(\sin_{pq}{x})'=\cos_{pq}^{\frac{q}{p}}{x},\notag \\
(\cos_{pq}{x})'=-\sin_{pq}^{q-1}{x}\cos_{pq}^{1-\frac{q}{p'}}{x}, \notag \\
(\cos_{pq}^{\frac{q}{p'}}{x})'=-\frac{q}{p'}\sin_{pq}^{q-1}{x}, \notag \\
(\tan_{pq}{x})'=\cos_{pq}^{-1-\frac{q}{p'}}{x}. \notag
\end{gather}
The case $p=q=r$ for some $r \in (1,\infty)$ is as follows.
\begin{gather*}
\cos_{r}^r{x}+\sin_{r}^r{x}=1,\label{eq:csr} \\
(\sin_{r}{x})'=\cos_{r}{x},\notag \\
(\cos_{r}{x})'=-\sin_{r}^{r-1}{x}\cos_{r}^{2-r}{x}, \notag \\
(\cos_{r}^{r-1}{x})'=-(r-1)\sin_{r}^{r-1}{x}, \notag \\
(\tan_{r}{x})'=\cos_{r}^{-r}{x}. \notag
\end{gather*}
It is useful to collect formulae for case $p=r'$ and $q=r$ for some $r \in (1,\infty)$.
\begin{gather}
\cos_{r'r}^r{x}+\sin_{r'r}^r{x}=1, \notag \\
(\sin_{r'r}{x})'=\cos_{r'r}^{r-1}{x}, \label{eq:sr'r} \\
(\cos_{r'r}{x})'=-\sin_{r'r}^{r-1}{x}, \label{eq:cr'r} \\
(\tan_{r'r}{x})'=\cos_{r'r}^{-2}{x}. \notag
\end{gather}
 
In particular, it is important that for any $p,\ q \in (1,\infty)$  
\begin{equation}
\label{eq:reduction}
((\sin_{pq}{x})')^p+\sin_{pq}^q{x}=1.
\end{equation}
We can find many other properties of these functions in 
\cite{EGL,LE}. 

\begin{rem}
There are some different definitions of $\cos_{pq}{x}$ from 
\eqref{eq:cos}. For example,
Dr\'{a}bek and Man\'{a}sevich \cite{DM}
define $\cos_{pq}{x}$ by 
$$\cos_{pq}{x}=(\sin_{pq}{x})',$$
and so \eqref{eq:reduction} gives
$$\cos_{pq}^p{x}+\sin_{pq}^q{x}=1,$$
which is slightly different from \eqref{eq:cs}.
The fact that $\sin_{pq}{x}$ satisfies \eqref{eq:reduction}
is essential, independently of the definition of $\cos_{pq}{x}$.
\end{rem}

\begin{prop}
\label{lem:pipp}
The number $\pi_{r'r}$ is an increasing function in $r \in (1,\infty)$
and $\disp \lim_{r \to 1+0}\pi_{r'r}=2$ and 
$\disp \lim_{r \to +\infty}\pi_{r'r}=4$.
\end{prop}

\begin{proof}
Putting $1-t^r=s$ in the definition of $\pi_{r'r}$, we have
\begin{equation}
\label{eq:beta}
\pi_{r'r}=\frac{2}{r}B\left(\frac1r,\frac1r\right).
\end{equation}
It suffices to show that 
$tB(t,t)$ is decreasing on $(0,1)$. Using a formula of Euler
(see Example 36 in \cite[p.262]{WW}):
$$\log{B(s,t)}=\log\frac{s+t}{st}+\int_0^1 \frac{(1-v^s)(1-v^t)}{(1-v)\log{v}}\,dv, \quad s,\ t \in (0,\infty),$$
we have
\begin{align*}
\log{tB(t,t)}
&=\log{t}+\log{B(t,t)}\\
&=\log{2}+\int_0^1 \frac{(1-v^t)^2}{(1-v)\log{v}}\,dv,
\quad t \in (0,\infty).
\end{align*}
Clearly, the right-hand side is decreasing in $t$ (note that $\log{v}<0$), 
so that $tB(t,t)$ is also decreasing on $(0,1)$. Furthermore, since
$$\lim_{t \to 1}tB(t,t)=B(1,1)=1 \quad \mbox{and} \quad 
\lim_{t \to 0}\log{tB(t,t)}=\log{2},$$
we obtain the values of limits.
\end{proof}

\begin{rem}
We can find another proof of Proposition \ref{lem:pipp} 
in \cite[Lemma 2.4]{EGL}, in which they use the fact that 
the area of $r$-circle $|x|^r+|y|^r=1$ is $\pi_{r'r}$
(see also \cite{LP}). 
\end{rem}

\subsection{Generalized Jacobian Elliptic Functions}
\label{ssec:GJEF}

In the fashion of the classical Jacobian elliptic functions,
we define new transcendental functions.

Let $p,\ q \in (1,\infty)$.
For any $k \in [0,1)$ we define $K_{pq}(k)$ by
\begin{equation}
\label{eq:maru1}
K_{pq}(k)=\int_0^1 \frac{dt}{(1-t^q)^{\frac1p}(1-k^qt^q)^{\frac{1}{p'}}}.
\end{equation}
Then, for any $k \in [0,1)$ and $x \in [0,K_{pq}(k)]$ 
we define $\sn_{pq}{(x,k)}$ by
\begin{equation}
\label{eq:maru2}
x=\int_0^{\sn_{pq}{(x,k)}} \frac{dt}{(1-t^q)^{\frac1p}(1-k^qt^q)^{\frac{1}{p'}}}.
\end{equation}
Clearly, $\sn_{pq}{(x,k)}$ is an increasing function in $x$ 
from $[0,K_{pq}(k)]$ onto $[0,1]$. 
We extend the domain of $\sn_{pq}{(x,k)}$ to $[0,2K_{pq}(k)]$
by $\sn_{pq}(x,k)=\sn_{pq}(2K_{pq}(k)-x,k)$, and furthermore,
to $\mathbb{R}$ by 
$\sn_{pq}{(x+2K_{pq}(k),k)}=-\sn_{pq}{(x,k)}$, 
so that $\sn_{pq}{(x,k)}$ has
$4K_{pq}(k)$-periodicity. 
We can see that $K_{22}(k)=K(k),\ \sn_{22}{(x,k)}=\sn{(x,k)},\ 
K_{pq}(0)=\pi_{pq}/2$, and $\sn_{pq}{(x,0)}=\sin_{pq}{x}$. Moreover, 
the function $y=\sn_{pq}{(x,k)}$ satisfies that $y,\ |y'|^{p-2}y' \in C^1(\mathbb{R})$,
and $y \in C^2(\mathbb{R})$ if $1<p \leq 2$.

Using $\sn_{pq}{(x,k)}$, for $x \in [0,K_{pq}(k)]$ we also define 
\begin{align*}
\cn_{pq}{(x,k)}&=(1-\sn_{pq}^q{(x,k)})^{\frac1q},\\
\dn_{pq}{(x,k)}&=(1-k^q\sn_{pq}^q{(x,k)})^{\frac1q}.
\end{align*}
Clearly, $\cn_{pq}{(x,k)}$ and $\dn_{pq}{(x,k)}$ are decreasing 
functions in $x$ from $[0,K_{pq}(k)]$ onto $[0,1]$. 
We extend the domains of $\cn_{pq}{(x,k)}$ and $\dn_{pq}{(x,k)}$ 
to $[-K_{pq}(k),K_{pq}(k)]$ in the same way of $\cos_{pq}{x}$, 
and furthermore,
to $\mathbb{R}$ by 
$\cn_{pq}{(x+2K_{pq}(k),k)}=-\cn_{pq}{(x,k)}$ and $\dn_{pq}{(x+2K_{pq}(k),k)}=\dn_{pq}{(x,k)}$, respectively.
This implies that $\cn_{pq}{(x,k)}$ and $\dn_{pq}{(x,k)}$ have $4K_{pq}(k)$- and $2K_{pq}(k)$-periodicity. 
We can see that $\cn_{pq}{(x,0)}=\cos_{pq}{x}$ and $\dn_{pq}{(x,0)}=1$.

These functions satisfy, for $x \in (0,K_{pq}(k))$
\begin{gather*}
\cn_{pq}^q{(x,k)}+\sn_{pq}^q{(x,k)}=1,\\ 
\dn_{pq}^q{(x,k)}+k^q \sn_{pq}^q{(x,k)}=1,\\
(\sn_{pq}{(x,k)})'=\cn_{pq}^{\frac{q}{p}}{(x,k)}\dn_{pq}^{\frac{q}{p'}}{(x,k)},\\
(\cn_{pq}{(x,k)})'=-\sn_{pq}^{q-1}{(x,k)}\cn_{pq}^{1-\frac{q}{p'}}(x,k)
\dn_{pq}^{\frac{q}{p'}}{(x,k)},\\
(\dn_{pq}{(x,k)})'=-k^q\sn_{pq}^{q-1}{(x,k)}\cn_{pq}^{\frac{q}{p}}{(x,k)}
\dn_{pq}^{1-\frac{q}{p}}(x,k).
\end{gather*}
In case $p=q=r$ for some $r \in (1,\infty)$, we write 
$\sn_{pq}(x,k),\ \cn_{pq}(x,k)$ and $\dn_{pq}(x,k)$ by
$\sn_{r}(x,k),\ \cn_{r}(x,k)$ and $\dn_{r}(x,k)$, respectively.
For the case, the formulae above become as follows.
\begin{gather*}
\cn_{r}^r{(x,k)}+\sn_{r}^r{(x,k)}=1,\\ 
\dn_{r}^r{(x,k)}+k^r \sn_{r}^r{(x,k)}=1,\\
(\sn_{r}{(x,k)})'=\cn_{r}{(x,k)}\dn_{r}^{r-1}{(x,k)},\\
(\cn_{r}{(x,k)})'=-\sn_{r}^{r-1}{(x,k)}\cn_{r}^{2-r}(x,k)
\dn_{r}^{r-1}{(x,k)},\\
(\dn_{r}{(x,k)})'=-k^r\sn_{r}^{r-1}{(x,k)}\cn_{r}^{r-1}{(x,k)}.
\end{gather*}
We also state the case $p=r'$ and $q=r$ for some $r \in (1,\infty)$.
\begin{gather*}
\cn_{r'r}^r{(x,k)}+\sn_{r'r}^r{(x,k)}=1,\\ 
\dn_{r'r}^r{(x,k)}+k^r \sn_{r'r}^r{(x,k)}=1,\\
(\sn_{r'r}{(x,k)})'=\cn_{r'r}^{r-1}{(x,k)}\dn_{r'r}{(x,k)},\\
(\cn_{r'r}{(x,k)})'=-\sn_{r'r}^{r-1}{(x,k)}\dn_{r'r}{(x,k)},\\
(\dn_{r'r}{(x,k)})'=-k^r\sn_{r'r}^{r-1}{(x,k)}\cn_{r'r}^{r-1}{(x,k)}
\dn_{r'r}^{2-r}(x,k).
\end{gather*}

Moreover, $y=\sn_{pq}{(x,k)}$ satisfies
\begin{equation*}
(|u'|^{p-2}u')'+\frac{(p-1)q}{p}|u|^{q-2}u
(1+(p-1)k^q-pk^q|u|^q)(1-k^q|u|^q)^{p-2}=0.
\end{equation*}

As mentioned in Introduction, 
the author \cite{T} has introduced another generalized  
Jacobian elliptic functions,
which also include both the Jacobian elliptic functions
and the generalized trigonometric functions.
However, we should note that the definitions above 
of $K_{pq}(k)$ and $\sn_{pq}(x,k)$ are slightly different from those of \cite{T},
in which the common index of $1-k^qt^q$ to \eqref{eq:maru1} and \eqref{eq:maru2}
is not $1/p'$ but $1/p$.
On account of the index, 
$K_{pq}(k)$ has similar asymptotic behavior near $k=1$ as 
$K(k)$, indeed, 
$\disp \lim_{k \to 1} K_{pq}(k)=\infty$ for any $p,\ q \in (1,\infty)$.

To observe the convergence properties 
of generalized Jacobian elliptic functions
as $k \to 1$, we will prepare \textit{generalized hyperbolic functions},
for which similar definitions are seen in \cite{Li}.

For $x \in [0,\infty)$, we define $\sinh_{pq}{x}$ by
\begin{equation}
\label{eq:sinhpq}
x=\int_0^{\sinh_{pq}{x}} \frac{dt}{(1+t^q)^{\frac1p}},
\end{equation}
and extend its domain to $\mathbb{R}$ by $\sinh_{pq}{x}=-\sinh_{pq}{(-x)}$.
Using $\sinh_{pq}{x}$, for $x \in [0,\infty)$, we define
$$\cosh_{pq}{x}=(1+\sinh_{pq}^q{x})^{\frac1q},$$
and extend its domain to $\mathbb{R}$ by $\cosh_{pq}{x}=\cosh_{pq}(-x)$.
The function $\tanh_{pq}{x}$ is defined by
$$\tanh_{pq}{x}=\frac{\sinh_{pq}{x}}{\cosh_{pq}{x}}.$$
We agree that $\sinh_p{x},\ \cosh_p{x}$ and $\tanh_p{x}$ denote
$\sinh_{pp}{x},\ \cosh_{pp}{x}$ and $\tanh_{pp}{x}$ when $p=q$,
respectively.
Putting $p=q$ and $t^p=s^p/(1-s^p)$ in \eqref{eq:sinhpq}, we have
$$x=\int_0^{\tanh_p{x}} \frac{dt}{1-t^p}.$$

Then, it is easy to prove the following properties: 
for any $p,\ q \in (1,\infty)$ and all $x \in \mathbb{R}$,
\begin{gather*}
\lim_{k \to 1} \sn_{pq}(x,k)=\tanh_q{x},\\
\lim_{k \to 1} \cn_{pq}(x,k)=\lim_{k \to 1} \dn_{pq}(x,k)=\frac{1}{\cosh_q{x}}.
\end{gather*}


\section{Properties of $\sn_{pq}(x,k)$ and $K_{pq}(k)$}
\label{sec:PsK}

In this section we observe some properties of generalized Jacobian elliptic 
function $\sn_{pq}(x,k)$ and its 
quarter period $K_{pq}(k)$. 

The function $y=\sn_{pq}(x,k)$ satisfies
that $\sn_{pq}(0,k)=0,\ \sn_{pq}(K_{pq}(k),k)=1$,
$0<\sn_{pq}(x,k)<1$ for $x \in (0,K_{pq}(k))$, 
$y \in C^1[0,K_{pq}(k)]$, and  
$$y'=(1-y^q)^{\frac1p}(1-k^qy^q)^{\frac{1}{p'}} \ge 0.$$
If $1<p \leq 2$, then $y \in C^2[0,K_{pq}(k)]$ and
$$y''=-\frac{q}{p} y^{q-1}(1-y^q)^{\frac2p-1}(1-k^qy^q)^{\frac{2}{p'}-1}
\left((1-k^q)+pk^q(1-y^q)\right) \le 0.$$
When $p>2$, we see that $y'' \in L^1(0,K_{pq}(k))$ and
\begin{equation}
\label{eq:integrability}
\int_0^{K_{pq}(k)} |y''|\,dx=-\int_0^{K_{pq}(k)}y''\,dx
=-[y']_0^{K_{pq}(k)}=1.
\end{equation}

To obtain the estimate of $K_{r'r}(k)$ in Lemma \ref{lem:estimatek} below, 
we state Tchebycheff's integral inequality
in \cite{M,QH}.
 
\begin{lem}
\label{lem:tchebycheff}
Let $f,\ g:[a,b] \to \mathbb{R}$ be integrable functions, both increasing
or both decreasing. Furthermore, let $p:[a,b] \to \mathbb{R}$ be a positive,
integrable function. Then
\begin{equation}
\label{eq:tchebycheff}
\int_a^b p(x)f(x)\,dx \int_a^b p(x)g(x)\,dx
\leq \int_a^b p(x)\,dx \int_a^b p(x)f(x)g(x)\,dx.
\end{equation}
If one of the functions $f$ or $g$ is nonincreasing and the other
nondecreasing, then the inequality in \eqref{eq:tchebycheff} is reversed.
\end{lem} 

Concerning the following Lemmas \ref{lem:monotonicity}--\ref{lem:homo}, 
we can refer to \cite{BE,EGL,LE} for 
the corresponding results of $\sin_{pq}{x}$ and $\pi_{pq}$, 
which are in case $k=0$ of $\sn_{pq}(x,k)$ and $K_{pq}(k)$. 
Lemma \ref{lem:estimatek} extends an estimate of $K(k)$ by
Qi and Huang \cite[Eq.(10)]{QH} to that of $K_{r'r}(k)$ for any $r \in (1,\infty)$.

\subsection{Properties of the Function $\sn_{pq}(x,k)$}

Since $f_n(x,k)=f_1(nx,k)$, it suffices to observe properties of 
$f_1(x,k)=\sn_{pq}(2K_{pq}(k)x,k)$ in order to study 
those of \eqref{eq:functions}.



\begin{lem}
\label{lem:monotonicity}
For each $x \in [0,1]$, 

$p \mapsto \sn_{pq}(2K_{pq}(k)x,k)$ is decreasing on $(1,\infty)$ for any fixed 
$k \in [0,1),\ q \in (1,\infty)$.

$q \mapsto \sn_{pq}(2K_{pq}(k)x,k)$ is decreasing on $(1,\infty)$ for any fixed 
$k \in [0,1),\ p \in (1,\infty)$.

$k \mapsto \sn_{pq}(2K_{pq}(k)x,k)$ is increasing on $[0,1)$
for any fixed $p,\ q \in (1,\infty)$.
\end{lem}

\begin{proof}
First we will show that $\sn_{pq}(2K_{pq}(k)x,k)$ is decreasing in 
$p \in (1,\infty)$.

Let $1<p<r<\infty$. Putting
$$f(t)=\frac{\sn_{rq}^{-1}(t,k)}{\sn_{pq}^{-1}(t,k)}
 \quad \mbox{for}\ t \in (0,1],$$ 
we have
$$f'(t)=\frac{G(t)}{(\sn_{pq}^{-1}(t,k))^2 (1-t^q)^{\frac1r}(1-k^qt^q)^{\frac{1}{r'}}},$$
where
$$G(t)=\sn_{pq}^{-1}(t,k)-g(t)\sn_{rq}^{-1}(t,k)$$
and 
$$g(t)=(1-t^q)^{\frac1r-\frac1p}(1-k^qt^q)^{\frac{1}{r'}-\frac{1}{p'}}
=\left(\frac{1-k^qt^q}{1-t^q}\right)^{\frac1p-\frac1r}.$$
Since $g(t)$ is increasing in $t \in (0,1)$ when $p<r$, 
it is easy to see that
$$G'(t)=-g'(t)\sn_{rq}^{-1}(t,k)<0,$$
which means that $G(t)<0$, i.e, $f'(t)<0$ for each $t \in (0,1]$. Thus,
$$\frac{K_{rq}(k)}{K_{pq}(k)} \le \frac{\sn_{rq}^{-1}(t,k)}{\sn_{pq}^{-1}(t,k)}
<1 \quad \mbox{for}\ t \in (0,1],$$
namely,
\begin{gather*}
\sn_{rq}^{-1}(t,k)<\sn_{pq}^{-1}(t,k),\\
\frac{1}{K_{pq}(k)}\sn_{pq}^{-1}(t,k) 
\le \frac{1}{K_{rq}(k)}\sn_{rq}^{-1}(t,k)
 \quad \mbox{for}\ t \in (0,1].
\end{gather*}
Therefore we conclude that
\begin{equation*}
\sn_{pq}(2K_{pq}(k)x,k) \ge \sn_{rq}(2K_{rq}(k)x,k)
 \quad \mbox{for}\ x \in [0,1],
\end{equation*}
so that $\sn_{pq}(2K_{pq}(k)x,k)$
is decreasing in $p>1$.

The assertions for $q$ and $k$ are proved in a similar way. 
It is enough to replace $f(t)$ by
$$\frac{\sn_{pr}^{-1}(t,k)}{\sn_{pq}^{-1}(t,k)}\ (1<q<r) \quad \mbox{and} \quad
\frac{\sn_{pq}^{-1}(t,l)}{\sn_{pq}^{-1}(t,k)}\ (0\leq k<l<1),$$
respectively, and to replace $g(t)$ by
$$\left(\frac{1-t^r}{1-t^q}\right)^{\frac1p}
\left(\frac{1-k^rt^r}{1-k^qt^q}\right)^{\frac{1}{p'}}
\quad \mbox{and} \quad
\left(\frac{1-l^qt^q}{1-k^qt^q}\right)^{\frac{1}{p'}},$$
respectively.
\end{proof}

\begin{lem}
\label{lem:convexity}
Let $p,\ q \in (1,\infty)$ and $k \in [0,1)$. Then, for any 
$x \in (0,K_{pq}(k)]$,
$$\frac{1}{K_{pq}(k)} \leq \frac{\sn_{pq}(x,k)}{x}<1.$$
\end{lem}

\begin{proof}
Let $y=\sn_{pq}(x,k)$. Putting $t=ys$ in \eqref{eq:maru2}, we have
$$\sn_{pq}^{-1}(y,k)=\int_0^1 \frac{y\,ds}{(1-y^qs^q)^{\frac1p}(1-k^qy^qs^q)^{\frac{1}{p'}}},$$
so that 
$$x=\sn_{pq}(x,k) \int_0^1 
\frac{ds}{(1-s^q\sn_{pq}^q(x,k))^{\frac1p}
(1-k^qs^q\sn_{pq}^q(x,k))^{\frac{1}{p'}}}.$$
Since $0<\sn_{pq}(x,k) \le 1$, we obtain
$$\sn_{pq}(x,k) < x \leq K_{pq}(k)\sn_{pq}(x,k),$$
which implies the assertion.
\end{proof}





\subsection{Properties of the Number $K_{pq}(k)$}


\begin{lem}
\label{lem:kpq}
$p \mapsto K_{pq}(k)$ is decreasing on $(1,\infty)$ for any fixed 
$k \in [0,1),\ q \in (1,\infty)$.

$q \mapsto K_{pq}(k)$ is decreasing on $(1,\infty)$ for any fixed 
$k \in [0,1),\ p \in (1,\infty)$.

$k \mapsto K_{pq}(k)$ is increasing on $[0,1)$ and
$$K_{pq}(0)=\frac{\pi_{pq}}{2},\quad
\lim_{k \to 1}K_{pq}(k)=\infty$$
for any fixed $p,\ q \in (1,\infty)$.
\end{lem}

\begin{proof}
The assertion on $p$ immediately follows from
$$K_{pq}(k)=\int_0^1 \left(\frac{1-k^qt^q}{1-t^q}\right)^{\frac1p}\frac{1}{1-k^qt^q}\,dt.$$
The remaining parts also follow from the form of $K_{pq}(k)$.
\end{proof}

\begin{lem}
\label{lem:symmetry}
Let $p,\ q \in (1,\infty)$ and $k \in [0,1)$. Then
$$K_{pq}(k)=\frac{p'}{q}K_{q'p'}(k^{\frac{q}{p'}}).$$
\end{lem}

\begin{proof}
Putting $1-t^q=x^{p'}$ in \eqref{eq:maru1}, we have
$$K_{pq}(k)=\frac{p'}{q(1-k^q)^{\frac{1}{p'}}}
\int_0^1 \frac{dx}{(1-x^{p'})^{\frac{1}{q'}}(1+\kappa x^{p'})^{\frac{1}{p'}}},$$
where $\kappa=k^q/(1-k^q)$.
Furthermore, setting $x=t(1+\kappa (1-t^{p'}))^{-\frac{1}{p'}}$, we obtain
\begin{align*}
K_{pq}(k)
&=\frac{p'}{q(1-k^q)^{1-\frac{1}{q'}}}
\int_0^1 \frac{dt}{(1-t^q)^{\frac{1}{q'}}(1+\kappa (1-t^{p'}))^{1-\frac{1}{q'}}}\\
&=\frac{p'}{q}\int_0^1 \frac{dt}{(1-t^{p'})^{\frac{1}{q'}}(1-k^qt^{p'})^{\frac1q}}.
\end{align*}
The integration of the right-hand side is equal to $K_{q'p'}(k^{\frac{q}{p'}})$. 
\end{proof}


\begin{lem}
\label{lem:homo}
Let $p,\ q \in (1,\infty)$ and $k \in [0,1)$. Then
\begin{gather*}
K_{pq}(k) \leq \frac{r}{q}K_{r'r}(k^{\frac{q}{r}}),
\end{gather*}
where $r=\max\{p',q\}$.
\end{lem}

\begin{proof}
Case $p' \leq q$ follows from only Lemma \ref{lem:monotonicity}.
Case $p' > q$ is also proved similarly after using Lemma \ref{lem:symmetry}.  
\end{proof}

\begin{lem}
\label{lem:estimatek}
Let $r \in (1,\infty)$ and $k \in (0,1)$. Then
\begin{equation}
\frac{\pi_{r'r}}{2}\frac{\sin^{-1}_r{k}}{k}
\leq K_{r'r}(k)
\leq \frac{\pi_{r'r}}{2}\frac{\tanh^{-1}_r{k}}{k}
\leq \frac{\pi_{r'r}}{2}\frac{1}{(1-k^r)^{\frac1r}}.
\label{eq:estimatek}
\end{equation}
\end{lem}

\begin{proof}
Putting $t=\sin_{r'r}{\theta}$ in the definition of $K_{r'r}(k)$
and using \eqref{eq:reduction}, we have
$$K_{r'r}(k)=\int_0^{\frac{\pi_{r'r}}{2}}
\frac{d\theta}{(1-k^r\sin_{r'r}^r{\theta})^{\frac1r}}.$$

We follow the proof of Qi and Huang \cite{QH}, which deals with
the case $r=2$. 
Let $p(x)=1,\ f(x)=(1-k^r\sin_{r'r}^r{\theta})^{-\frac1r},\
g(x)=\cos_{r'r}^{r-1}{\theta}$ or $\sin_{r'r}^{r-1}{\theta},\
[a,b]=[0,\pi_{r'r}/2]$ in Tchebycheff's integral inequality
\eqref{eq:tchebycheff} of 
Lemma \ref{lem:tchebycheff}, then we obtain
$$K_{r'r}(k)
\int_0^{\frac{\pi_{r'r}}{2}} \cos_{r'r}^{r-1}{\theta}\,d\theta
\geq \int_0^{\frac{\pi_{r'r}}{2}} 1\,d\theta
\int_0^{\frac{\pi_{r'r}}{2}}
\frac{\cos_{r'r}^{r-1}{\theta}}{(1-k^r\sin_{r'r}^r{\theta})^{\frac1r}}\,d\theta$$
or
$$K_{r'r}(k)
\int_0^{\frac{\pi_{r'r}}{2}} \sin_{r'r}^{r-1}{\theta}\,d\theta
\leq \int_0^{\frac{\pi_{r'r}}{2}} 1\,d\theta
\int_0^{\frac{\pi_{r'r}}{2}}
\frac{\sin_{r'r}^{r-1}{\theta}}{(1-k^r\sin_{r'r}^r{\theta})^{\frac1r}}\,d\theta.$$
By \eqref{eq:sr'r} and \eqref{eq:cr'r}, easy calculation gives
$$\int_0^{\frac{\pi_{r'r}}{2}} \cos_{r'r}^{r-1}{\theta}\,d\theta
=\int_0^{\frac{\pi_{r'r}}{2}} \sin_{r'r}^{r-1}{\theta}\,d\theta
=1.$$
Moreover, putting $\sin_{r'r}{\theta}=t$, we have
\begin{align*}
\int_0^{\frac{\pi_{r'r}}{2}}
\frac{\cos_{r'r}^{r-1}{\theta}}{(1-k^r\sin_{r'r}^r{\theta})^{\frac1r}}\,d\theta
=\int_0^1 \frac{dt}{(1-k^rt^r)^{\frac1r}}
=\frac1k \int_0^k \frac{ds}{(1-s^r)^{\frac1r}}
=\frac1k \sin^{-1}_r{k}.
\end{align*}
Similarly, putting $\cos_{r'r}{\theta}=t$ and $t^r=\frac{1-k^r}{k^r}\frac{s^r}{1-s^r}$, we obtain
\begin{align*}
\int_0^{\frac{\pi_{r'r}}{2}}
\frac{\sin_{r'r}^{r-1}{\theta}}{(1-k^r\sin_{r'r}^r{\theta})^{\frac1r}}\,d\theta
=\int_0^1 \frac{dt}{(1-k^r+k^rt^r)^{\frac1r}}
=\frac1k \int_0^k \frac{ds}{1-s^r}=\frac1k \tanh_r^{-1}{k}.
\end{align*}
Thus, we accomplished the first and second inequalities of \eqref{eq:estimatek}.

Finally, from the equality above, we obtain the third inequality of 
\eqref{eq:estimatek} as
$$\frac1k \tanh_r^{-1}{k}
=\int_0^1 \frac{dt}{(1-k^r+k^rt^r)^{\frac1r}}
\leq \frac{1}{(1-k^r)^{\frac1r}}.$$

The graphs of terms of \eqref{eq:estimatek} for $r=2$ 
can be shown in Figure \ref{fig:kpq}.
\end{proof}

\begin{figure}[htbp]
\centering
\includegraphics[width=5cm,clip]{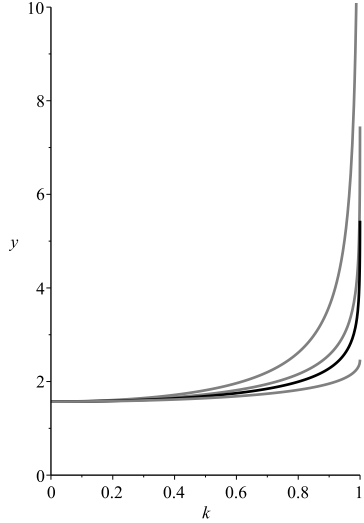}
\caption{The graphs of terms of \eqref{eq:estimatek} for $r=2$.
The black line and the gray lines indicate $K_{r'r}(k)$ and the others,
respectively.}
\label{fig:kpq}
\end{figure}


\section{The Operator $T$}
\label{sec:TOT}

Let $\alpha \in (1,\infty)$ be an arbitrary number. 
In this section, we will make the functions
$$f_n(x)=f_n(x,k)=\sn_{pq}(2nK_{pq}(k)x,k) \in L^\alpha(0,1), \quad n=1,2,\ldots,$$
correspond to the sine series
$$e_n(x)=\sin(n \pi x) \in L^\alpha(0,1),\quad n=1,2,\ldots,$$
which form a basis of $L^\alpha(0,1)$.

\begin{thm}
\label{thm:homeomorphism}
If there is a linear homeomorphism $T$ of $L^\alpha(0,1)$ satisfying 
$Te_n=f_n,\ n=1,2,\ldots$,
then $\{f_n(x,k)\}$ forms a Riesz basis of $L^2(0,1)$ and 
a Schauder basis of $L^\alpha(0,1)$ for any 
$\alpha \in (1,\infty)$.
\end{thm}

\begin{proof}
By Example \ref{exmp:singer343}, any function in $L^\alpha(-1,1)$
has a unique sine-cosine series representation. For any $f \in L^\alpha(0,1)$,
we can thus represent its odd extension to $L^\alpha(-1,1)$ uniquely in a sine
series, so the $e_n$ form a basis of $L^\alpha(0,1)$. 
Since $\{e_n\}$ and $\{f_n\}$ are equivalent, according to Proposition
\ref{prop:higgins75} the same is true for the $f_n$.
It follows from Proposition \ref{prop:schauder} that  
they form a Schauder basis of $L^\alpha(0,1)$.
The argument for a Riesz basis when $\alpha=2$ is similar and follows 
from Proposition \ref{thm:bari311}. 
\end{proof}

In the remainder of this section we define $T$ as a linear combination
of certain isometries of $L^\alpha(0,1)$. Then we show that $T$ is a bounded
operator satisfying $Te_n=f_n,\ n=1,2,\ldots$, for all $p,\ q \in (1,\infty)$.

The functions $f_n$
have Fourier sine series expansions
$$f_n(x)=\sum_{l=1}^\infty \widehat{f_n}(l) e_l(x),$$
where 
$$\widehat{f_n}(l)=2\int_0^1 f_n(x) e_l(x)\,dx.$$
An argument involving symmetry with respect to the middlepoint 
$x=1/2$ easily shows that $\widehat{f_1}(l)=0$ whenever
$l$ is even. On account of this property, 
we can show $\widehat{f_n}(l)$ by using $\widehat{f_1}(l)$
as follows.
\begin{align}
\label{eq:(10)}
\widehat{f_n}(l)
&=2\int_0^1 f_1(nx)e_l(x)\,dx \notag \\
&=2\sum_{m:\mathrm{odd}}^\infty \widehat{f_1}(m)\int_0^1 e_{mn}(x)
e_l(x)\,dx \notag \\
&=
\begin{cases}
\widehat{f_1}(m) & \mbox{if $mn=l$ for some odd $m$},\\
0 & \mbox{otherwise}.
\end{cases}
\end{align}

In what follows we will often denote $\widehat{f_1}(m)$ by $\tau_m$.
We first find a bound on $|\tau_m|$ which will be crucial in the 
definition of $T$ below. Since $\tau_m=0$ if $m$ is even, 
we may assume that $m$ is odd. Integration by parts ensures that
\begin{align*}
\tau_m
&=4\int_0^{\frac12} f_1(x)e_m(x)\,dx
=-\frac{4}{m^2\pi^2}\int_0^{\frac12}f_1''(x)e_m(x)\,dx,
\end{align*}
where the integrals exist because $f_1'' \in L^1(0,1)$.
In fact \eqref{eq:integrability} shows that
\begin{equation}
\label{eq:(12)}
|\tau_m|
\le \frac{4}{m^2\pi^2}
\int_0^{\frac12}|f_1''(x)e_m(x)|\,dx
<\frac{4}{m^2\pi^2} \int_0^{\frac12}|f_1''(x)|\,dx
=\frac{8K_{pq}(k)}{m^2\pi^2}.
\end{equation}

In order to construct the linear operator $T$,
we next define isometries $M_m$ of the Banach space $L^\alpha(0,1)$
by $M_mg(x):=g^*(mx),\ m=1,2,\ldots$,
where $g^*$ is its successive antiperiodic extension 
of $g$ over $\mathbb{R}_+$ by $g^*=g$ on $[0,1]$, and
\begin{equation*}
g^*(x)=-g^*(2n-x) \quad \mbox{if}\ n <x \le n+1, 
\quad n=1,2,\ldots.
\end{equation*}
Notice that $M_me_n=e_{mn}$.

\begin{lem}
\label{lem:7}
The maps $M_m$ are isometric linear transformations of $L^\alpha(0,1)$
for all $m=1,2,\ldots$ and $\alpha \in (1,\infty)$.
\end{lem}

\begin{proof}
\begin{align*}
\|M_mg\|_\alpha^\alpha
&=\int_0^1|M_mg(x)|^\alpha\,dx=\int_0^1 |g^*(mx)|^\alpha\,dx\\
&=\frac{1}{m}\int_0^m|g^*(u)|^\alpha\,du
=\frac{1}{m}\sum_{l=1}^m \int_{l-1}^l|g^*(u)|^\alpha\,du\\
&=\frac{1}{m}\sum_{l=1}^m \int_0^1|g(u)|^\alpha\,du
=\int_0^1|g(u)|^\alpha\,du=\|g\|_\alpha^\alpha.
\end{align*}
\end{proof}

We now define $T:L^\alpha(0,1) \to L^\alpha(0,1)$ by
\begin{equation}
\label{eq:(13)}
Tg(x)=\sum_{m=1}^\infty \tau_m M_m g(x).
\end{equation}

Lemma \ref{lem:7}, the triangle inequality and \eqref{eq:(12)} ensure
that $T$ is a bounded everywhere-defined operator with
$\|T\|_{(L^\alpha \to L^\alpha)} \le K_{pq}(k)$.

We conclude this section by showing that $Te_n=f_n$.
Indeed, by virtue of \eqref{eq:(10)},
$$Te_n=\sum_{m=1}^\infty \tau_m e_{mn}
=\sum_{m:\mathrm{odd}}^\infty \widehat{f_1}(m)e_{mn}
=\sum_{l=1}^\infty \widehat{f_n}(l)e_{l}
=f_n.$$


\section{Bounded Invertibility of $T$}
\label{sec:BIT}

In this section we complete the proof of Theorem \ref{thm:Main}
by showing that $T$ has a bounded inverse for all $p,\ q$
and $k$ satisfying \eqref{eq:Alpha} and \eqref{eq:K}.

Observe that \eqref{eq:(13)} and the triangle inequality give
$$\|Tg-\tau_1M_1g\|_\alpha\le \sum_{m=3}^\infty |\tau_m| \|M_m g\|_\alpha,$$
so that Lemma \ref{lem:7} and the fact $M_1=I$
(the identity operator on $L^\alpha(0,1)$) give
$$\|T-\tau_1I\|_{(L^\alpha \to L^\alpha)} \le \sum_{m=3}^\infty |\tau_m|.$$
Thus, by C.\,Neumann's theorem \cite[Theorem 2 in p.\,69]{Y},
to show that $T$ has a bounded inverse,
it is sufficient to prove 
\begin{equation}
\label{eq:(14)}
\sum_{m=3}^\infty |\tau_m| < |\tau_1|.
\end{equation}

We estimate the left-hand side from above and 
the right-hand side from below. Inequality \eqref{eq:(12)} shows that
\begin{equation}
\label{eq:lhs}
\sum_{m=3}^\infty |\tau_m| \leq \frac{8K_{pq}(k)}{\pi^2}
\left(\frac{\pi^2}{8}-1\right).
\end{equation}
On the other hand, it follows from Lemma \ref{lem:convexity} that
\begin{equation}
\label{eq:rhs}
|\tau_1|=4\int_0^{\frac12}\sn_{pq}(2K_{pq}(k)x,k)\sin{\pi x}\,dx
\geq 8\int_0^{\frac12}x\sin{\pi x}\,dx=\frac{8}{\pi^2}.
\end{equation}
Then, we can prove

\begin{thm}
\label{thm:general}
Let $p,\ q \in (1,\infty)$ and $k \in [0,1)$. The sequence  
$\{f_n(x,k)\}$ forms a Riesz basis of $L^2(0,1)$ and 
a Schauder basis of $L^\alpha(0,1)$ for any 
$\alpha \in (1,\infty)$ if
\begin{equation}
\label{eq:kpq}
K_{pq}(k) < \frac{8}{\pi^2-8}.
\end{equation}
\end{thm}

\begin{rem}
In case $k=0$, inequality \eqref{eq:kpq} corresponds with
$\pi_{pq}<16/(\pi^2-8)$, which is given in \cite[Corollary 4.3]{EGL}.
\end{rem}

\begin{proof}
Combining \eqref{eq:lhs}, \eqref{eq:kpq} and \eqref{eq:rhs}, 
we obtain \eqref{eq:(14)}. Thus, $T$ has a bounded inverse
and is a linear homeomorphism in $L^\alpha(0,1)$. Therefore,
from Theorem \ref{thm:homeomorphism}, 
we have completed the proof.
\end{proof}

Now we are in a position to show Theorem \ref{thm:Main}.

\begin{proof}[Proof of Theorem $\ref{thm:Main}$]
Assume \eqref{eq:Alpha} and 
let $k$ be a number satisfying \eqref{eq:K}. 
From Lemmas \ref{lem:homo} and \ref{lem:estimatek}, 
we have
$$K_{pq}(k) \leq
\frac{r \pi_{r'r}}{2q}
\frac{\tanh_r^{-1}{k^{\frac{q}{r}}}}{k^{\frac{q}{r}}}.$$
Using \eqref{eq:beta} to the right-hand side, we obtain
$$K_{pq}(k) \leq \frac{1}{q}B\left(\frac{1}{r},\frac{1}{r}
\right)\frac{\tanh_r^{-1}{k^{\frac{q}{r}}}}{k^{\frac{q}{r}}}.$$
Thus, \eqref{eq:Alpha} and \eqref{eq:K} give
$$K_{pq}(k)<\frac{8}{\pi^2-8},$$
which implies \eqref{eq:kpq} of Theorem \ref{thm:general}.
We have thus proved the theorem.
\end{proof}

\begin{proof}[Proof of Corollary $\ref{cor:main}$]
Let $1<p' \le q<\infty$. Then $r=q$, and it suffices to show 
that \eqref{eq:Alpha} is satisfied.   
Since we have the inequality $tB(t,t) \leq 2$ in 
the proof of Proposition \ref{lem:pipp},
we obtain
$$\frac{1}{q}B\left(\frac{1}{r},\frac{1}{r}\right)
\leq 2<\frac{8}{\pi^2-8},$$
so that $\eqref{eq:Alpha}$ holds. 

In particular, we consider the case $p=q=2$. 
Then, $r=q=2$, and by Theorem \ref{thm:Main},
the sequence $\{\sn(2nK(k)x,k)\}$ is a basis
in $L^\alpha(0,1)$ for any $\alpha \in (1,\infty)$ when
$$\frac{\tanh^{-1}{k}}{k}
\leq \frac{16}{(\pi^2-8)\pi},$$
which holds true if $0<k \leq 0.9909\cdots$.
\end{proof}

\begin{proof}[Proof of Corollary $\ref{thm:main}$]
Suppose that $q$ and $r$ satisfy \eqref{eq:alpha}. 
Since $tB(t,t) \leq 2$ as the proof of Corollary \ref{cor:main},
we obtain
$$\frac{1}{q}B\left(\frac{1}{r},\frac{1}{r}\right)
\leq \frac{2r}{q}<\frac{8}{\pi^2-8},$$
which implies \eqref{eq:Alpha}. 
Furthermore, 
if we assume \eqref{eq:k}, then \eqref{eq:estimatek} in 
Lemma \ref{lem:estimatek} shows
$$\frac{\tanh_r^{-1}{k^{\frac{q}{r}}}}{k^{\frac{q}{r}}}
\leq \frac{1}{(1-k^q)^{\frac1r}}
\leq \frac{4q}{(\pi^2-8)r}
\leq \frac{8q}{\pi^2-8}B\left(\frac{1}{r},\frac{1}{r}\right)^{-1},$$
so that \eqref{eq:K} holds, and Theorem \ref{thm:Main} gives 
Corollary \ref{thm:main}.  
\end{proof}

\end{document}